\documentclass[12pt]{amsart}
\usepackage{amsmath, amsthm, amscd, amsfonts}

\usepackage{lineno}

\makeatletter \oddsidemargin.9375in \evensidemargin \oddsidemargin
\newtheorem{theorem}{Theorem}[section]
\newtheorem{lemma}[theorem]{Lemma}

\theoremstyle{definition}
\newtheorem{definition}[theorem]{Definition}
\newtheorem{example}[theorem]{Example}

\theoremstyle{remark}

\numberwithin{equation}{section}

\newcommand{\N}{\mathbb{N}}
\newcommand{\R}{\mathbb{R}}
\begin{document}

\title[Pursuit differential game with geometric and
integral constraints] {On a fixed duration pursuit differential
game\\ with geometric and integral constraints}

\author[Mehdi Salimi, Gafurjan I.\ Ibragimov, Stefan Siegmund, Somayeh Sharifi]{Mehdi Salimi$^{\ast \dag}$, Gafurjan I.\ Ibragimov$^{\ddag}$, Stefan Siegmund$^{\dag}$, Somayeh Sharifi$^{\S}$\\
\\
$^{\dag}$Center for Dynamics, Department of Mathematics,\\
Technische Universit{\"a}t Dresden, 01062 Dresden, Germany\\
\\
$^{\ddag}$Department of Mathematics \& Institute for Mathematical
Research, Universiti Putra Malaysia, 43400 UPM Serdang,
 Selangor, Malaysia\\
 \\
$^{\S}$Young Researchers and Elite Club, Hamedan Branch, Islamic
Azad University, Hamedan, Iran}

 \keywords{differential game; pursuer;
evader; strategy; value of the game.\\
\indent $^{\ast}$corresponding author.\\
\indent \textit{E-mail addresses:} mehdi.salimi@tu-dresden.de\\
}
\begin{abstract}
In this paper we investigate a differential game in which
countably many dynamical objects pursue a single one. All the
players perform simple motions. The duration of the game is fixed.
The controls of a group of pursuers are subject to geometric
constraints and the controls of the other pursuers and the evader
are subject to integral constraints. The payoff of the game is the
distance between the evader and the closest pursuer when the game
is terminated. We construct optimal strategies for players and
find the value of the game.
\end{abstract} \maketitle

\section{Introduction and Preliminaries}
Differential game theory comes into play when one wants to study
procedures in which one controlled object is pursued by others.
There are several types of differential games, but the most common
one is the so called pursuit-evasion game. Fundamental researches
by Isaacs \cite{isa}, Krasovskii \cite{kra},  Pashkov and
Teorekhov \cite{pas}, Petrosyan \cite{pet},  Pontryagin
\cite{pon}, Rikhsiev \cite{rik} and Rzymowski \cite{rzy} deal with
differential games and pursuit evasion problems. Constructing the
player's optimal strategies and finding the value of the game are
of special interest in the study of differential games. Such
problems in case of many pursuers were studied for example, by
some of the authors in \cite{ibr0}-\cite{ibr2} and by Ivanov et
al. in \cite{iva}.\\

Pursuit-evasion games have several applications in robotics such
as motion planning in adversarial settings (e.g. playing
hide-and-seek) or defining the requirements to achieve a goal in
the worst-case performance of robotic systems. As an example
imagine a search-and-rescue setting in which the goal is to find a
lost person by robots. Treating the problem as a pursuit-evasion
game, a pursuit strategy (upon existence) guarantees the rescue of
the lost person regardless of his/her movements. Therefore, a
worst-case bound on the number of robots required to rescue a lost
person can be obtained by considering the person as an adversarial
entity trying to evade capture.

In \cite{ibr1} Ibragimov studies a differential game of optimal
approach of countably many pursuers to one evader in a Hilbert
space with geometric constraints on the controls of the players.
Ibragimov and Salimi \cite{ibr2} study such a differential game
for inertial players with integral constraints under the
assumption that the control resource of the evader is less than
that of each pursuer. Evasion from many pursuers in simple motion
differential games with integral constraints was investigated by
Ibragimov et al.\ in \cite{ibr3} as well.

In the present paper, motivated by the above developments, we
discuss an optimal pursuit problem with countably many pursuers
and one evader in the Hilbert space $\ell_2 = \{\alpha =
(\alpha_k)_{k \in \N} \in \R^\N :
\sum_{k=1}^\infty\alpha_k^2<\infty\}$ with inner product
$(\alpha,\beta)=\sum_{k=1}^\infty\alpha_k\beta_k$ and norm
$\|\alpha\|_2 := \|\alpha\|_{\ell_2} = \big(
\sum_{k=1}^\infty\alpha_k^2 \big )^{\frac{1}{2}}$. The controls of
a group of pursuers are subject to geometric constraints and the
controls of the other pursuers and the evader are subject to
integral constraints. In section 2 we formulate the problem, give
dynamic equations and basic definitions. In section 3 we introduce
an auxiliary game and introduce admissible strategy of the pursuer
that guaranties it to capture evader, and in section 4 we point
out the main theorem that estimate the value of the game. An
example to illustrate the theorem conclude the section.

\section{Formulation of the Problem}

Let $\theta > 0$ be arbitrary, but fixed. For $1 \leq p \leq \infty$ let ${L_p} = L_{p}([0,\theta], \ell_2) = \{ f : [0, \theta] \rightarrow \ell_2 \,|\, f \text{ is measurable and } \|f\|_{L_p} < \infty\}$ denote the space of $L_p$ functions from $[0,\theta]$ to $\ell_2$ with norm $\|f\|_p := \|f\|_{L_p} = \big( \int_0^\theta \|f(t)\|_{\ell_2}^p \, dt \big)^{\frac{1}{p}}$ for $1 \leq p < \infty$ and $\|f\|_\infty := \|f\|_{L_\infty} = \operatorname{ess\, sup}_{t \in [0, \theta]} \|f(t)\|_{\ell_2}$. Then for $1 \leq p < q \leq \infty$ the inclusions $L_q \subset L_p$ hold.
For arbitrary \emph{controls} $v \in L_2$, $u_n \in L_2$ ($n \in \N$), and \emph{initial values} $y_0 \in \ell_2$, $x_{n0} \in \ell_2$ ($n \in \N$), consider the infinitely many initial value problems
\begin{equation}\label{aa}
  \begin{array}{llrl}
   (P_n):& \dot x_n=u_n(t),& x_n(0)=x_{n0},& n \in \N ,
  \\[1ex]
    (E):& \dot y~=v(t), & y(0)~=~y_0,&
  \end{array}
\end{equation}
on the interval $[0, \theta]$. The solution $x_n$ of $(P_n)$ is an
element of the space $AC = AC([0,\theta], \ell_2)$ of absolutely
continuous functions from $[0,\theta]$ to $\ell_2$. It is given by
\[
  x_n = \big( x_{nk} \big)_{k \in \N}
  \qquad \text{with} \quad
  x_{nk}(t) = x_{nk0} + \int_0^t u_{nk}(s) \, ds \quad
  \text{for}\quad t \in [0, \theta]
\]
and is called \emph{motion} of the \emph{pursuer $P_n$}. The
solution $y \in AC$ of $(E)$ is
\[
  y = \big( y_k \big)_{k \in \N}
  \qquad \text{with} \quad
  y_{k}(t) = y_{k0} + \int_0^t v_{k}(s) \, ds \quad
  \text{ for } \quad t \in [0, \theta]
\]
and is called \emph{motion} of the \emph{evader $E$}.

For a Banach space $(X, \|\cdot\|_X)$ let $B_X(x_0,r)=\{x\in X : \|x-x_0\| \le r\}$ and $S_X(x_0,r) = \{x \in X : \|x-x_0\|=r\}$ denote the ball and sphere of radius $r$ and center $x_0$, respectively.


\begin{definition}[Admissible controls]
Let $I, J \subset \N$ with $I \cup J = \N$ and $\rho_n > 0$ for $n
\in \N$. A function $u_n \in L_2$ is called \emph{admissible
control for the pursuer} $P_n$ if it satisfies

(a) \emph{Integral constraint for pursuers on $I$:} If $n \in I$
then $u_n \in \mathcal{U}_{c}^{(n)} := B_{L_2}(0,\rho_n)$, i.e.
\[
  \Big( \int_0^\theta \sum_{k=1}^{\infty} u_{nk}(t)^{2} \,dt \Big)^{\frac{1}{2}}
  \leq
  \rho_n
  .
\]

(b) \emph{Geometric constraint for pursuers on $J$:} If $n \in J$ then $u_n \in \mathcal{U}_{c}^{(n)} := B_{L_\infty}(0,\rho_n)$, i.e.
\[
  \max_{t \in [0,\theta]} \Big( \sum_{k=1}^{\infty} u_{nk}(t)^{2} \Big)^{\frac{1}{2}}
  \leq
  \rho_n
  .
\]
Note that w.l.o.g.\ we replaced the essential supremum over $[0,
\theta]$ by $\max$.

$\mathcal{U}_{c}^{(n)}$ for $n\in \mathbb{N}$ is called the
\emph{set of all admissible controls of pursuer $P_n$}.

Let $\sigma> 0$. A function $v \in L_2$ is called \emph{admissible
control for the evader} if it satisfies

(c) \emph{Integral constraint for evader:} $v \in \mathcal{V}_c := B_{L_2}(0,\sigma)$, i.e.
\[
  \Big( \int_0^\theta \sum_{k=1}^{\infty} v_{k}(t)^{2} \,dt \Big)^{\frac{1}{2}}
  \leq
  \sigma.
\]
$\mathcal{V}_{c}$ is called the \emph{set of all admissible controls of evader $E$}.
\end{definition}


\begin{definition}[Admissible strategies]

(a) A function
\begin{equation*}
  U_n : [0,\infty) \times \ell_2 \times \ell_2\times \ell_2 \to \ell_2,
  \quad
  (t,x_n,y,v) \mapsto U_n(t,x_n,y,v)
 \qquad \text{for } n\in \mathbb{N},
\end{equation*}
is called \emph{strategy of the pursuer} $P_n$ if for any
admissible control $v \in {\mathcal V}_c$ of the evader $E$ and
any initial values $y_0 \in \ell_2$, $x_{n0} \in \ell_2$, the
system of equations
\begin{equation}\label{eq:strategypursuer}
  \begin{array}{lr}
   \dot x_n=U_n(t,x_n,y,v(t)),& x_n(0)=x_{n0},
  \\[1ex]
   \dot y~=v(t), & y(0)~=~y_0,
  \end{array}
\end{equation}
has a unique solution $(x_n,y)$ with $x_n$, $y \in AC$. A strategy
$U_n$ is said to be \textit{admissible} if for every solution
$(x_n,y)$ the control $u_n(t) := U_n(t,x_n(t),y(t),v(t))$
generated by $U_n$ is admissible, i.e.\
$U_n(\cdot,x_n(\cdot),y(\cdot),v(\cdot)) \in {\mathcal
U}_c^{(n)}$. The set of all admissible strategies for
\eqref{eq:strategypursuer} is denoted by $\mathcal{U}^{(n)}$.

(b) A function
\begin{equation*}
  V : [0,\infty)\times \ell_2^{\N} \to \ell_2,
  \quad
  (t,x_1,\ldots,x_m,\ldots,y) \mapsto V(t,x_1,\ldots,x_m,\ldots,y)
\end{equation*}
is called \emph{strategy of the evader} $E$, if for arbitrary admissible
controls $u_n \in {\mathcal U}_c^{(n)}$ of the pursuers $P_n$ ($n \in \N$),
arbitrary initial values $y_0 \in \ell_2$, $x_{n0} \in \ell_2$, the system of equations
\begin{equation}\label{eq:strategyevader}
  \begin{array}{ll}
   \dot x_n= u_n(t),& x_n(0)=x_{n0}, \quad n \in \N,
  \\[1ex]
   \,\dot y~=V(t,x_1, \dots, x_m, \dots, y), & \,y(0)~=~y_0,
  \end{array}
\end{equation}
has a unique solution $(x_1,\ldots,x_m,\ldots,y)$ with $x_n, y \in
AC$ for $n \in \N$. A strategy $V$ is said to be
\textit{admissible}, if for every solution
$(x_1,\ldots,x_m,\ldots,y)$ the control $v(t) := V(t,x_1(t),
\dots, x_m(t), \dots, y(t))$ generated by $V$ is admissible, i.e.\
$V(\cdot,x_1(\cdot), \dots, x_m(\cdot), \dots, y(\cdot)) \in
\mathcal{V}_c$. The set of all admissible strategies for
\eqref{eq:strategyevader} is denoted by $\mathcal{V}$.
\end{definition}
\begin{definition}[Optimal strategies and value of the game]
(a)  Admissible strategies $U_{n}^*$ of the pursuers $P_n$ ($n \in \N$) are said to be
\emph{optimal} if
\begin{equation*}
  \inf_{\substack{U_n \in \mathcal{U}^{(n)}\\ \text{for } n \in \N}} \Gamma_1(U_1,\ldots,U_m,\ldots)
  =
  \Gamma_1(U_{1}^*,\ldots,U_{m}^*,\ldots),
\end{equation*}
where
\begin{equation*}
  \Gamma_1(U_1,\ldots,U_m,\ldots)
  =
  \sup_{v \in \mathcal{V}_c}\,\inf_{n\in \N} \|x_n(\theta)-y(\theta)\|
\end{equation*}
and $(x_n, y)$ is the solution of \eqref{eq:strategypursuer}.

(b) An admissible strategy $V^*$ of the evader $E$ is said to be \emph{optimal} if
\begin{equation*}
  \sup_{V \in \mathcal{V}} \Gamma_2(V)=\Gamma_2(V^*),
\end{equation*}
 where
\begin{equation*}
  \Gamma_2(V) = \inf_{\substack{u_n \in \mathcal{U}_c^{(n)}\\ \text{for } n \in \N}}
  \inf_{n\in \mathbb{N}}\|x_n(\theta)-y(\theta)\|,
\end{equation*}
and $(x_1,\ldots,x_m,\ldots,y)$ is the solution of
\eqref{eq:strategyevader}.

(c) If
$\Gamma_1(U_{1}^*,\ldots,U_{m}^*,\ldots)=\Gamma_2(V^*)=\gamma$
then we say that \emph{the game has the value $\gamma$}
\cite{sub}.
\end{definition}

Our aim is to find optimal strategies $U_{n}^*$ and $V^*$ of the players
$P_n$ and $E$, respectively, and the value of the game.

\section{An Auxiliary Game}

The attainability domain of the pursuers $P_i$ and $P_j$ at time $\theta$ from the initial state $x_{i0}$ and $x_{j0}$ are the closed balls $B_{\ell_2} (x_{i0},\rho_i\sqrt{\theta})$ for $i \in I$ and $B_{\ell_2}(x_{j0},\rho_j\theta)$ for $j \in J$, respectively. This is due to the estimates
\begin{equation}\label{11}
\begin{split}
  \|x_i(\theta)-x_{i0}\|
  &=
  \Big\|\int_0^\theta u_i(s) \,ds \Big\|
  \leq
  \int_0^\theta \|u_i(s)\| \, ds
  \\
  &\leq
  \Big(\int_{0}^{\theta}1 \,ds\Big)^{1/2}
  \cdot
  \Big(\int_{0}^{\theta} \|u_i(s)\|^2 \,ds \Big)^{1/2}
  \leq
  \rho_i\sqrt{\theta},
  \qquad \text{for } i\in I,
\end{split}
\end{equation}
and
\begin{equation}\label{22}
  \|x_j(\theta)-x_{j0}\|
  =
  \Big\|\int_0^\theta u_j(s) \, ds \Big\|
  \leq
  \int_0^\theta \|u_j(s)\| \,ds
  \leq
  \rho_j\theta,
  \qquad \text{for } j\in J.
\end{equation}
On the other hand, for $i \in I$ an arbitrary $\bar x \in B_{\ell_2}(x_{i0},\rho_i\sqrt{\theta})$ can be reached by a pursuer $P_i$ with the admissible control $u_i \in \mathcal{U}_c^{(i)}$ defined by
\begin{equation*}
  u_i(s)
  =
  (\bar x-x_{i0})/\theta, \quad 0\leq s \leq \theta,
\end{equation*}
which implies $x_i(\theta)=\bar x$. Moreover, for $j \in J$ an $\bar x \in B_{\ell_2}(x_{j0},\rho_j\theta)$ can be reached by choosing the admissible control
\begin{equation*}
  u_j(s)
  =
  (\bar x-x_{j0})/\theta, \quad 0\leq s \leq \theta,
\end{equation*}
for the pursuer $P_j$, which results in $x_j(\theta)=\bar x$.

Similarly, the attainability domain of the evader $E$ at time
$\theta$ from the initial state $y_0$ is the
closed ball $B_{\ell_2} (y_0,\sigma\sqrt{\theta})$.

For simplicity we consider now the
following game with only one pursuer
\begin{equation}\label{aaaa}
 \begin{split}
&(P):\dot{x}=u(t), \quad x(0)=x_{0},\\
&(E):\dot{y}=v(t), \quad y(0)=y_0,
 \end{split}
\end{equation}
and assume at first that the control $u$ satisfies an integral constraint $u \in B_{L_2}(0,\rho)$ for some $\rho > 0$.
Define
\begin{equation*}
  X
  =
  \Big\{ z \in \ell_2: 2(y_0-x_0,z) \leq (\rho^2-\sigma^2)\theta+\|y_0\|^2-\|x_0\|^2 \Big\}.
\end{equation*}

\begin{lemma}\label{110}
If $y(\theta)\in X$, then for the game \eqref{aaaa} with a pursuer which is subject to an integral constraint, there exists an admissible strategy of the pursuer $P$ with $x(\theta)=y(\theta)$.
\end{lemma}
\begin{proof}
We define the pursuer's strategy as follows:
\begin{equation}\label{ss}
  u(t)
  =
  \tfrac{1}{\theta}(y_0-x_0)+v(t), \quad 0\le t \le \theta.
\end{equation}
We show that if $y(\theta)\in X$, then the above strategy is
admissible. Using the fact that
\begin{equation*}
y(\theta)=y_0+\int_{0}^{\theta}v(s)ds,
\end{equation*}
then from the inequality
\begin{equation*}
2(y_0-x_0, y(\theta))\le
(\rho^2-\sigma^2)\theta+\|y_0\|^2-\|x_0\|^2,
\end{equation*}
and equality
\begin{equation*}
\|y_0-x_0\|^2+\|y_0\|^2-\|x_0\|^2-2(y_0-x_0,y_0)=0,
\end{equation*}
we obtain
\begin{equation*}
2\int_0^{\theta}(y_0-x_0, v(s))ds\le
(\rho^2-\sigma^2)\theta-\|x_0-y_0\|^2.
\end{equation*}
From $(\ref{ss})$, we get
\begin{equation*}
\begin{split}
\int_0^{\theta}\|u(s)\|^2ds&=\frac{1}{\theta}\|y_0-x_0\|^2+\frac{2}{\theta}\int_0^{\theta}(y_0-x_0,v(s))ds+\int_0^{\theta}\|v(s)\|^2ds\\
&\le
\frac{1}{\theta}\|y_0-x_0\|^2+\frac{1}{\theta}\left((\rho^2-\sigma^2)\theta-\|x_0-y_0\|^2\right)+\sigma^2=\rho^2,
\end{split}
\end{equation*}
therefore strategy $(\ref{ss})$ is admissible. Then
\begin{equation*}
  x(\theta)
  =
  x_0+\int_{0}^{\theta}u(s)ds=x_0+y_0-x_0+\int_0^{\theta}v(s)ds=y(\theta).
  \qedhere
\end{equation*}
\end{proof}
Now consider the second case and assume that the pursuer's control $u$ is subject to
a geometric constraint $u \in B_{L_\infty}(0,\rho)$ for some $\rho > 0$. We have the following lemma.
\begin{lemma}\label{1}
If $\sigma\le\rho$ and $y(\theta)\in X$, then for the game \eqref{aaaa} with a pursuer which is subject to a geometric constraint, there exists an admissible strategy of the pursuer $P$ with $x(\theta)=y(\theta)$.
\end{lemma}

\begin{proof}
We introduce the pursuer's strategy as follows:
\begin{equation}\label{d}
u(t)=
                \begin{cases}
                v(t)-(v(t),e)e+e \big(\rho^2-\sigma^2+(v(t),e)^{2}\big)^{1/2},
                \quad &\text{if $ \quad 0\le t\le\tau$},\\
                v(t),\quad &\text{if $\quad \tau<t\le\theta$},
\end{cases}
\end{equation}
where $e=(y_0-x_0)/(\|y_0-x_0\|)$ and $\tau \in [0, \theta]$
is the time instant at which $x(\tau)=y(\tau)$ for the first
time.

We have $y(t)-x(t)=ef(t)$, where
\begin{equation*}
f(t)=\|y_0-x_0\|+\int_0^{t}\left(v(s),e\right)ds-\int_0^{t}\left(\rho^2-\sigma^2+(v(s),e)^2\right)^{\frac{1}{2}}ds.
\end{equation*}
Clearly, $f(0)=\|y_0-x_0\|>0$.
Let us show that $f(\theta)\leq0$. This will imply that
$f(\tau)=0$ for some $\tau\in [0, \theta]$.

Consider the two-dimensional vector function
$g(t)=\big((\rho^2-\sigma^2)^{\frac{1}{2}}, (v(t),e)\big)$. We
get
\begin{equation*}
\begin{split}
  \int_0^{\theta}\left(\rho^2-\sigma^2+(v(t),e)^2\right)^{\frac{1}{2}}ds
  &=
  \int_0^{\theta}|g(s)|ds\geq\Big|\int_0^{\theta}g(s)ds\Big|
\\
  &=\Big((\rho^2-\sigma^2)\theta^2+\Big(\int_0^{\theta}(v(s),e)ds\Big)^2\Big)^{\frac{1}{2}}.
\end{split}
\end{equation*}
Then
\begin{equation}\label{aaa}
f(\theta)\leq
\|y_0-x_0\|+\int_0^{\theta}(v(s),e)ds-\Big((\rho^2-\sigma^2)\theta^2+
\Big(\int_0^{\theta}(v(s),e)ds\Big)^2\Big)^{\frac{1}{2}}.
\end{equation}
By assumption $y(\theta)\in X$. Consequently,
$(e,y(\theta))\leq d$, where
\begin{equation*}
d=\left(
(\rho^2-\sigma^2)\theta^2+\|y_0\|^2-\|x_0\|^2\right)/(2\|y_0-x_0\|).
\end{equation*}
Hence, we obtain
\begin{equation}\label{bbb}
\int_0^{\theta}(v(s), e)ds\leq d-(y_0, e).
\end{equation}
On the other hand the function
$\phi(t)=\|y_0-x_0\|+t-\left((\rho^2-\sigma^2)\theta^2+t^2\right)^{\frac{1}{2}}$
is increasing on $\R$. Then it follows from
(\ref{aaa}) and (\ref{bbb}) that
\begin{equation*}
f(\theta)\leq
\|y_0-x_0\|+d-(y_0,e)-\left((\rho^2-\sigma^2)\theta^2+(d-(y_0,e))^2\right)^{\frac{1}{2}}=0.
\end{equation*}
Consequently, $f(\tau)=0$ for some $\tau \in [0,\theta]$.
Therefore, $x(\tau)=y(\tau)$. Further, by (\ref{d}), $u(t)=v(t)$
for $ \tau \leq t \leq \theta$. Then, obviously,
$x(\theta)=y(\theta)$.
\end{proof}

\section{Main Result}

Now consider the game $(\ref{aa})$. We will solve the optimal pursuit problem under the following
assumption.

\textbf{Assumption (A)} \cite{ibr1}. \textit{There exists a
nonzero vector $p_0$ such that $(y_0-x_{n0},p_0)\ge0$ for all
$n\in \mathbb{N}.$}

For $l \geq 0$ define
\begin{equation}
\begin{split}
  G_i(l)
  &:=
  B_{\ell_2}\big(x_{i0},\rho_i\sqrt{\theta}+l\big),\quad
  \text{for } i\in I,
\\
  G_j(l)
  &:=
  B_{\ell_2}\big(x_{j0},\rho_j\theta+l\big),\qquad \text{for } j\in J,
\end{split}
\end{equation}
therefore
\begin{equation}\label{h2}
  \gamma
  =
  \inf \Big\{ l \ge 0 : B_{\ell_2} \big(y_0,\sigma\sqrt{\theta}\big) \subset \bigcup_{n \in \N} G_n(l)\Big\}.
\end{equation}
\begin{theorem}\label{theorem}
Under the assumption (A) and if
$\sigma\theta\le\rho_j\theta+\gamma$ for all $j\in J,$ then the
number $\gamma$ given by (\ref{h2}) is the value of game
(\ref{aa}).
\end{theorem}

The proof of Theorem \ref{theorem} relies on the following lemmas for which we let $y_0, x_{n0} \in \ell_2$ with $x_{n0}\neq y_0$ for $n \in \N$ and choose $r, R_n > 0$ for $n \in \N$.
\begin{lemma}[\cite{ibr1}]\label{2}
Let
\begin{equation*}
X_n=\Big\{z\in l_2:2(y_0-x_{n0},z)\le
R_n^2-r^2+\|y_0\|^2-\|x_{n0}\|^2\Big\}.
\end{equation*}
Under the assumption (A) and if
\begin{equation*}
  B_{\ell_2}(y_0,r) \subset
  \bigcup_{n\in \N} B_{\ell_2}( x_{n0},R_n ),
\end{equation*}
then $B_{\ell_2}(y_0,r) \subset \bigcup_{n\in \N} X_n$.
\end{lemma}
\begin{lemma}[\cite{ibr1}]\label{3}
Let $R_0 := \inf_{n\in \mathbb{N}}R_n$. Under the assumption (A) and if $R_0 > 0$ and if for any $\varepsilon \in (0,R_0)$ the set $\bigcup_{n\in \mathbb{N}} B_{\ell_2}(x_{n0},R_n-\varepsilon)$ does not contain the ball
$B_{\ell_2}(y_0,r)$, then there exists a point $\bar y\in S_{\ell_2}(y_0,r)$ such
that $\|\bar y-x_{n0}\|\ge R_n$ for all $n\in \N$.
\end{lemma}

\begin{proof}[Proof of Theorem \ref{theorem}]
We assume that
\begin{equation}
  K
  =
  \big\{ n \in \N : S_{\ell_2}(y_0, \sigma\sqrt{\theta}) \cap G_n(\gamma) \neq \emptyset \big\},
\end{equation}
and consider the infinitely many initial value problems
\begin{equation}\label{aaaaa}
  \begin{array}{llrl}
   (P_k):& \dot x_k=u_k(t),& x_k(0)=x_{k0},& k \in K,
  \\[1ex]
    (P_{\hat{k}}):& \dot x_{\hat{k}}=0,& & {\hat k} \in \N \setminus  K,
    \\[1ex]
    (E):& \dot y=v(t),&y(0)=y_0,&
  \end{array}
\end{equation}
on the interval $[0, \theta]$. The value $\gamma$ of this game is the same
as for the game (\ref{aa}). Indeed by the definition of $\gamma$, we
have
\begin{equation*}\label{q1}
  B_{\ell_2}\big(y_0,\sigma\sqrt{\theta}\big)
  \subset
  \bigcup_{n\in \N} G_n(\gamma),
\end{equation*}
therefore
\begin{equation*}\label{q2}
  B_{\ell_2}\big(y_0,\sigma\sqrt{\theta}\big)
  \subset
  \bigcup_{n\in K} G_n(\gamma).
\end{equation*}
For arbitrary $\varepsilon > 0$ we introduce fictitious pursuers $z_n,$ whose motions are
described by the equations
\begin{equation*}
  \dot{z}_n
  =
  w_n^\varepsilon, \quad z_n(0)=x_{n0}, \qquad \text{for } n \in \N=I\cup J,
\end{equation*}
with integral constraint on $I$
\begin{equation}\label{pp}
\begin{split}
  & \|w^{\varepsilon}_i\|_{L_2}
  =
  \Big(\int_0^\theta\|w_i^{\varepsilon}(t)\|^2\,dt \Big)^{1/2}
  \le
  \bar{\rho_i}(\varepsilon)
  :=
  \rho_i+\frac{\gamma}{\sqrt{\theta}}+\frac{\varepsilon}{k_i\sqrt{\theta}},
  \qquad \text{for } i\in I,
\end{split}
\end{equation}
and with geometric constraint on $J$
\begin{equation}\label{ppp}
\begin{split}
  &\|w^{\varepsilon}_j\|_{L_\infty}
  =
  \max_{t \in [0, \theta]}
  \Big(\sum_{k=1}^{\infty}w_{jk}(t)^2 \Big)^{1/2}
  \le
  \bar{\rho_j}(\varepsilon)
  :=
  \rho_j+\frac{\gamma}{\theta}+\frac{\varepsilon }{k_j\theta},
  \qquad \text{for } j \in J,
\end{split}
\end{equation}
where $k_i=\max\{1,\rho_i\},$ and $k_j=\max\{1,\rho_j\}$. It is obvious
that the attainability domains of the fictitious pursuers $z_i$, $i \in I$,
and $z_j$, $j \in J$, at time $\theta$ from the initial states $x_{i0}$ and
$x_{j0}$ are the balls
$B_{\ell_2}(x_{i0},\bar{\rho_i}(\varepsilon)\sqrt{\theta})
=B_{\ell_2}(x_{i0},\rho_i\sqrt{\theta}+\gamma+\tfrac{\varepsilon}{k_i})$
and
$B_{\ell_2}(x_{j0},\bar{\rho_j}(\varepsilon)\theta)=
B_{\ell_2}(x_{j0},\rho_j\theta+\gamma+\tfrac{\varepsilon}{k_j})$.

It can be shown that the attainability domain of the fictitious
pursuers from the initial position $x_{n0}$ up to time $\theta$ is
a ball $G_n(\gamma)$. We define $\widetilde{I}=I\cap K$ and
$\widetilde{J}=J\cap K$ and strategies of the fictitious pursuers
$z_i$ for $i\in \widetilde{I}$ as follows:
\begin{equation}
w_i^{\varepsilon}(t)=\begin{cases}\label{kkk}
                       \dfrac{1}{\theta}(y_0-x_{i0})+v(t), &\text{if $0\leq t\leq
                       \tau_i^{\varepsilon}$},\\
                       0, &\text{if $\tau_i^{\varepsilon}< t\leq \theta$},
                       \end{cases}
\end{equation}
where $\tau_i^{\varepsilon} \in [0, \theta]$ is the time for which
\begin{equation}
\int_0^{\tau_i^{\varepsilon}}\|w_i^{\varepsilon}(s)\|^2ds=\bar\rho_i^2(\varepsilon),
\end{equation}
if such a time exists. Since $\bar\rho_i(\varepsilon)>\bar\rho_i(0):=\bar\rho_i$, we get
\begin{equation*}
\int_0^{\tau_i^{\varepsilon}}\|w_i^{\varepsilon}(s)\|^2ds=\bar\rho_i^2(\varepsilon)>\bar\rho_i^2=\int_0^{\tau_i}\|w_i^0(s)\|^2ds:=\int_0^{\tau_i}\|w_i(s)\|^2ds,
\end{equation*}
that is,
\begin{equation*}
  \int_0^{\tau_i^{\varepsilon}}\big\|\frac{y_0-x_{i0}}{\theta}+v(t)\big\|^2\,ds
  >
  \int_0^{\tau_i}\big\|\frac{y_0-x_{i0}}{\theta}+v(t)\big\|^2\,ds,
\end{equation*}
hence $\tau_i^{\varepsilon}>\tau_i$.

Also we define the strategies of the fictitious pursuers  $z_j$
for $j\in \widetilde{J}$ as follows:
\begin{equation}
w_j^{\varepsilon}(t)=\begin{cases}\label{kkkk}
                       v(t)-\left(v(t),e_j\right)e_j+e_j\left(\bar\rho_j^2(\varepsilon)-\sigma^2+\left(v(t),e_j\right)^2\right)^{1/2}, \quad &\text{if \quad $0\le t\le \tau_j$},\\
                       v(t), \quad &\text{if \quad $\tau_j\le t\le \theta$},
                       \end{cases}
\end{equation}
where $e_j=(y_0-x_{j0})/\|y_0-x_{j0}\|$ and $\tau_j \in [0, \theta]$ is the time at which
$z_j(\tau_j)=y(\tau_j)$ for the first time, if it exists.

Now we define the strategies of the pursuers $x_i$, $i\in
\widetilde{I}$, and $x_j$, $j\in \widetilde{J}$, by the strategies
of the fictitious pursuers as follows:
\begin{equation}\label{oo}
u_n(t)=\frac{\rho_n \theta^\xi}{\rho_n\theta^\xi+\gamma}w_n(t),
\quad \quad 0\le t \le \theta,
\end{equation}
\begin{equation*}
  \text{with} \quad \xi =
  \left\{
  \begin{array}{ll}
    \frac{1}{2}, & \text{if } n \in \tilde{I},
  \\
    1, & \text{if } n \in \tilde{J},
  \end{array}
  \right.
\end{equation*}
as well as $u_n(t)=0$ for $t \in [0, \theta]$ and $n\in \N\setminus K$.

We show now that constructed strategies $u_i(t)=\dfrac{\rho_i\sqrt{\theta}}{\rho_i\sqrt{\theta}+\gamma}w_i(t)$ in  \eqref{oo} for the pursuers $P_i$, $i\in \widetilde{I}$, satisfy the inequalities
\begin{equation}\label{m}
  \sup_{v \in \mathcal{V}_c} \inf_{i\in \widetilde{I}} \|y(\theta)-x_i(\theta)\|
  \le
  \gamma.
\end{equation}
By the definition of $\gamma$, we have
$$
  B_{\ell_2}\big(y_0,\sigma\sqrt{\theta}\big)
  \subset
  \bigcup_{i\in \widetilde{I}}
  B_{\ell_2}\big(x_{i0},\rho_i\sqrt{\theta}+\gamma+\frac{\varepsilon}{k_i}\big).
$$
By assumption, the inequality $(y_0-x_{i0},p_0)\ge0$ holds for
all $i\in \widetilde{I}.$ Then it follows from Lemma \ref{2} that
\begin{equation*}
  B_{\ell_2}\big(y_0,\sigma\sqrt{\theta}\big)
  \subset
  \bigcup_{i\in \widetilde{I}} X_i^\varepsilon,
\end{equation*}
where
\begin{equation*}
  X_i^\varepsilon
  =
  \Big\{z\in \ell_2 :
  2(y_0-x_{i0},z) \le
  \big(\rho_i\sqrt{\theta}+\gamma+\frac{\varepsilon}{k_i}\big)^2 - \sigma^2\theta+\|y_0\|^2-\|x_{i0}\|^2\Big\}.
\end{equation*}
Consequently, the point $y(\theta)\in B_{\ell_2}(y_0,
\sigma\sqrt{\theta})$ belongs to some half-space
$X_s^{\varepsilon}, s=s(\varepsilon)\in \widetilde{I}$, and we
have
\begin{equation}\label{ddd}
  2(y_0-x_{s0}, y(\theta))
  \leq
  \big(\rho_s\sqrt{\theta}+\gamma+\frac{\varepsilon}{k_s}\big)^2-\sigma^2\theta+\|y_0\|^2-\|x_{s0}\|^2.
\end{equation}
By Lemma \ref{110} for the strategies (\ref{kkk}) of fictitious
pursuers we get $z_s(\theta)=y(\theta)$. Then taking into account
(\ref{pp}) and (\ref{oo}) with $\xi=\frac{1}{2}$, we get
\begin{equation}\label{nn}
 \begin{split}
  \|y(\theta)-x_s(\theta)\|
  &=
  \|z_s(\theta)-x_s(\theta)\|
  =
  \Big\|\,\int_0^\theta\left(w_s^\varepsilon(t)-\frac{\rho_s}{\bar\rho_s}w_s(t)\right)dt\Big\|
\\
  &\le
  \int_0^\theta\|w_s^\varepsilon(t)-w_s(t)\|\,dt+\int\limits_0^\theta\|w_s(t)-\frac{\rho_s}{\bar\rho_s}w_s(t)\|\,dt.
 \end{split}
\end{equation}
Now we put aside the right-hand side of the last inequality. Let
us show that
\begin{equation}\label{o}
\int\limits_0^\theta\|w_i^\varepsilon(t)-w_i(t)\|dt\leq
K_1\sqrt{\varepsilon},\quad  \text{for all} \quad i\in
\widetilde{I}.
\end{equation}
Indeed, as we show that $\tau_i^{\varepsilon}>\tau_i$ and
according to (\ref{kkk}) $w_i^\varepsilon(t)=w_i(t)$ for $0\le
t\le \tau_i$, $w_i(t)=0$ for $t>\tau_i$, $w_i^\varepsilon(t)=0$
for $t>\tau_i^{\varepsilon}$, then we have
\begin{equation*}
\begin{split}
  \int_0^\theta\|w_i^{\varepsilon}(t)-w_i(t)\|dt
  &=
  \int_0^{\tau_i}\|w_i^{\varepsilon}(t)-w_i(t)\|dt +
  \int_{\tau_i}^{\tau_i^{\varepsilon}}\|w_i^\varepsilon(t)-w_i(t)\|dt
\\
  &{} + \int_{\tau_i^{\varepsilon}}^\theta\|w_i^{\varepsilon}(t)-w_i(t)\|dt
\\
  &=
  \int_{\tau_i}^{\tau_i^{\varepsilon}}\|w_i^{\varepsilon}(t)\|dt
  \leq
  \sqrt{\tau_i^{\varepsilon}-\tau_i}\Big(\int_{\tau_i}^{\tau_i^{\varepsilon}}\|w_i^{\varepsilon}(t)\|^2dt\Big)^{1/2}
\\
  &\leq
  \sqrt{\theta}\Big(\int_{0}^{\tau_i^{\varepsilon}}\|w_i^{\varepsilon}(t)\|^2dt -
  \int_{0}^{\tau_i}\|w_i^{\varepsilon}(t)\|^2dt\Big)^{1/2}
\\
  &=
  \sqrt{\theta}\Big(\frac{\varepsilon}{k_i\sqrt{\theta}}\Big(2\rho_i+\frac{2\gamma}{\sqrt{\theta}} +
  \frac{\varepsilon}{k_i\sqrt{\theta}}\Big)\Big)^{1/2}
\\
  &\leq
  K_1\sqrt{\varepsilon},
\end{split}
\end{equation*}
where $K_1$ is some positive number.

For the second integral in (\ref{nn}) we have
\begin{equation*}
\begin{split}
  \Big\|\int_0^{\theta}\Big(1-\frac{\rho_s}{\bar\rho_s}\Big)w_s(t)\Big\|
  &\leq
  \Big(1-\frac{\rho_s}{\bar\rho_s}\Big) \int_0^{\theta}\|w_s(t)\|dt
\\
  &\leq
  \Big(1-\frac{\rho_s}{\bar\rho_s}\Big)
  \Big(\int_0^\theta dt\Big)^{1/2}
  \Big( \int_0^\theta\|\,w_s(t)\|^{2}dt\Big)^{1/2}
\\
  &\leq
  \Big(1-\frac{\rho_s}{\bar\rho_s}\Big)\sqrt{\theta}\bar\rho_s=\gamma.
\end{split}
\end{equation*}
Then from (\ref{nn}) it follows that $\|y(\theta)-x_s(\theta)\|\leq
\gamma+K_1\sqrt{\varepsilon}$. Thus, if the pursuers use the
strategies (\ref{oo}) with $\xi=\frac{1}{2}$, the inequality (\ref{m}) holds.

Now we show that the strategies $u_j(t)=\dfrac{\rho_j\theta}{\rho_j\theta+\gamma}w_j(t)$ constructed in \eqref{oo} of the pursuers $P_j$ for $j\in \widetilde{J}$ satisfy the inequality
\begin{equation}\label{mm}
  \sup_{v \in \mathcal{V}_c} \inf_{j\in \widetilde{J}}
  \|y(\theta)-x_j(\theta)\|
  \le
  \gamma.
\end{equation}
By the definition of $\gamma$, we have
\begin{equation*}
  B_{\ell_2}\big(y_0,\sigma\sqrt{\theta}\big)
  \subset
  \bigcup_{j\in \widetilde{J}}
  B_{\ell_2}\big(x_{j0},\rho_j\theta+\gamma+\frac{\varepsilon}{k_j}\big).
\end{equation*}
By assumption, the inequality $(y_0-x_{j0},p_0)\ge0$ holds for
all $j\in \widetilde{J}.$ Then it follows from Lemma \ref{2} that
\begin{equation*}
  B_{\ell_2}\big(y_0,\sigma\sqrt{\theta}\big)
  \subset
  \bigcup_{j\in\widetilde{J}} X_j^\varepsilon,
\end{equation*}
where
\begin{equation*}
  X_j^\varepsilon
  =
  \Big\{ z \in \ell_2 : 2(y_0-x_{j0},z) \le
  \big(\rho_j\theta+\gamma+\frac{\varepsilon}{k_j}\big)^2-
  \sigma^2\theta+\|y_0\|^2-\|x_{j0}\|^2\Big\}.
\end{equation*}
Consequently, the point $y(\theta)\in
B(y_0,\sigma\sqrt{\theta})$ belongs to some half-space
$X_s^\varepsilon,$ $s=s(\varepsilon)\in \widetilde{J}.$ By the
assumption of the Theorem \ref{theorem},
$\bar\rho_j(\varepsilon)>\sigma;$ then it follows from Lemma
\ref{1} that if $z_j$ uses strategies $(\ref{kkkk})$ then
$z_s(\theta)=y(\theta).$ By taking account of $(\ref{oo})$ with
$\xi=1,$ we obtain
\begin{equation}\label{n}
  \begin{split}
  \|y(\theta) - x_s(\theta)\|
  &=
  \|z_s(\theta) - x_s(\theta)\|
  =\Big\|\int_0^\theta\big(w_s^\varepsilon(t)-\frac{\rho_s}{\bar\rho_s}w_s(t)\big)\,dt\Big\|
\\
  &\le
  \int_0^\theta\|w_s^\varepsilon(t)-w_s(t)\|\,dt +
  \int_0^\theta\|w_s(t)-\frac{\rho_s}{\bar\rho_s}w_s(t)\|\,dt.
 \end{split}
\end{equation}
Now we show that the first term of the right-hand side of the inequality satisfies
\begin{equation}\label{o}
\lim_{\varepsilon\to0}\sup_{j\in
\widetilde{J}}\int_0^\theta\|w_j^\varepsilon(t)-w_j(t)\|dt=0.
\end{equation}

 If there exists
$\tau_j\in[0,\theta],$ mentioned in (\ref{kkkk}), then
\begin{equation*}
 \begin{split}
   &
  \int_0^{\theta}\|w_j^\varepsilon(t)-w_j(t)\|dt
 \\
  &=
  \int_0^{\tau_j}\Big(\big( (\bar\rho_j^2(\varepsilon)-\sigma^2)+(v(t),e_j)^2\big)^{1/2} -
  \big((\bar\rho_j^2-\sigma^2)+(v(t),e_j)^2\big)^{1/2}\Big)^2\,dt
\\
  &\leq
  \int_0^{\tau_j}\big((\bar\rho_j^2(\varepsilon)-\sigma^2)^{1/2} -
  (\bar\rho_j^2-\sigma^2)^{1/2}\big)^2\,dt
\\
  &\leq
  \int_0^{\theta}\big((\bar\rho_j^2(\varepsilon)-\sigma^2)^{1/2} -
  (\bar\rho_j^2-\sigma^2)^{1/2}\big)^2\,dt
\\
  &=
  \theta
  \big((\bar\rho_j^2(\varepsilon)-\sigma^2)^{1/2} -
  (\bar\rho_j^2-\sigma^2)^{1/2}\big)^2
\\
  &\leq
  \theta
  \big(\big(\bar\rho_j^2-\sigma^2+\big(\frac{\varepsilon}{k_j\theta}\big)^2+
  \frac{2\bar\rho_j\varepsilon}{k_j \theta}\big)^{1/2}-
  (\bar\rho_j^2-\sigma^2)^{1/2}\big)^2
\\
  &\leq
  \theta
  \big(\big(\frac{\varepsilon}{k_j\theta}\big)^2+\frac{2\bar\rho_j\varepsilon}{k_j\theta}\big)
  \leq
  \varepsilon\big(\frac{\varepsilon}{\theta}+2\big(1+\frac{\gamma}{\theta}\big)\big)
  =
  K_2 \varepsilon,
  \end{split}
\end{equation*}
where $K_2$ is some positive number.

For the second integral in \eqref{n} we have
\begin{equation*}
  \begin{split}
  \int_0^\theta \big\| \big(1-\frac{\rho_s}{\bar\rho_s}\big) w_s(t) \big\|\,dt
  &=
  \big(1-\frac{\rho_s}{\bar\rho_s}\big)
  \int_0^\theta\|\,w_s(t)\|\, dt
\\
  &\le
  \big(1-\frac{\rho_s}{\bar\rho_s}\big)\bar\rho_s\theta
  =
  \gamma.
 \end{split}
\end{equation*}
Then it follows from $(\ref{n})$ that
$\|y(\theta)-x_s(\theta)\|\le\gamma+K_2\varepsilon$.
Thus if the pursuers use strategies $(\ref{oo})$ with $\xi=1,$
the inequality $(\ref{mm})$ holds.

We construct the evader's strategies ensuring that
\begin{equation}\label{p1}
\inf_{\substack{u_n \in \mathcal{U}_c^{(n)}\\ \text{for } n \in K}}
\inf_{i\in \widetilde{I}}\|y(\theta)-x_i(\theta)\|\ge\gamma,
\end{equation}
and
\begin{equation}\label{p2}
\inf_{\substack{u_n \in \mathcal{U}_c^{(n)}\\ \text{for } n \in K}}
\inf_{j\in \widetilde{J}}\|y(\theta)-x_j(\theta)\|\ge\gamma,
\end{equation}
First, if $\gamma=0,$ then inequality $(\ref{p1})$ is obviously
valid for any admissible control of the evader. Let $\gamma>0$. By
the definition of $\gamma,$ for any $\varepsilon>0,$ the set
\begin{equation*}
  \bigcup_{i\in \widetilde{I}}
  B_{\ell_2}\big(x_{i0},\rho_i\sqrt{\theta}+\gamma-\varepsilon\big),
\end{equation*}
does not contain the ball $B_{\ell_2}(y_0,\sigma\sqrt{\theta})$.
Then, by Lemma \ref{3} there exists a point $\bar y\in
S_{\ell_2}(y_0,\sigma\sqrt{\theta})$, such that $\|\bar
y-x_{i0}\|\ge\rho_i\sqrt{\theta}+\gamma.$ On the other hand
\begin{equation*}
\|x_i(\theta)-x_{i0}\|\leq\rho_i\sqrt{\theta}.
\end{equation*}
Consequently
\begin{equation*}
\|\bar y-x_i(\theta)\|\ge\|\bar
y-x_{i0}\|-\|x_i(\theta)-x_{i0}\|\ge\rho_i\sqrt{\theta}+\gamma-\rho_i\sqrt{\theta}=\gamma.
\end{equation*}
Then the value of the game is not less than $\gamma,$ and
inequality $(\ref{p1})$ holds.
Similarly these relationships can be  proved for $j\in \widetilde{J}$.

Therefore
\begin{equation*}
  \sup_{v \in \mathcal{V}_c} \inf_{n\in K} \|y(\theta)-x_n(\theta)\|
  \leq
  \gamma,
\end{equation*}
and
\begin{equation*}
  \inf_{\substack{u_n \in \mathcal{U}_c^{(n)}\\ \text{for } n \in K}}\,
  \inf_{n\in K}\|y(\theta)-x_n(\theta)\|
  \ge
  \gamma.
\end{equation*}
The proof of the theorem is complete.
\end{proof}
Now, we provide an example to illustrate Theorem \ref{theorem}.
\begin{example}
Let $\rho_i=2$, $\rho_j=1$, $\theta=9$ and $\sigma=2$. We consider
the following initial positions $ x_{i0}=(0,\ldots,0,3,0,\ldots)$,
$ x_{j0}=(0,\ldots,0,8,0,\ldots)$ and $y_0=(0,0,\ldots)$ of the
players, where number $3$ is $i$th coordinate of the point
$x_{i0}$ and number $8$ is $j$th coordinate of the point $x_{j0}$.
Let us obtain the value of the
game. To this end, it is sufficient to show that\\
1) for any $\varepsilon>0$ the inclusion $ B_{\ell_2}(O,6)\subset
\bigcup_{i=1}^{\infty}B_{\ell_2}(x_{i0},6.7+\frac{\varepsilon}{k_i})$
holds, where $O$
is the origin and $k_i=\max\{1,\rho_i\}=2$.\\
$\text{1}^\star$) for any $\varepsilon>0$ the inclusion $
B_{\ell_2}(O,6)\subset
\bigcup_{j=1}^{\infty}B_{\ell_2}(x_{j0},10+\frac{\varepsilon}{k_j})$ holds, where $k_j=\max\{1,\rho_j\}=1$.\\
2) the ball $B_{\ell_2}(O,6)$ is not contained in the set
$\bigcup_{i=1}^{\infty}B_{\ell_2}(x_{i0},6.7)$.\\
$\text{2}^\star$) the ball $B_{\ell_2}(O,6)$ is not contained in
the set $\bigcup_{j=1}^{\infty}B_{\ell_2}(x_{j0},10)$.

Let $z=(z_1,z_2,\ldots)$ be an arbitrary point of the ball
$B_{\ell_2}(O,6)$. So $\sum_{i=1}^{\infty}z_i^2\leq36$ and
$\sum_{j=1}^{\infty}z_j^2\leq36$. Then, either the vector $z$ has
a nonnegative coordinate or all the coordinate of the vector $z$
are negative. In the former case, $z$ has a nonnegative coordinate
$z_k$. Then,
\begin{equation*}
\begin{split}
\|z-x_{k0}\|&=\left(z_1^2+\ldots+z_{k-1}^2+(3-z_k)^2+z_{k+1}^2+\ldots\right)^{1/2}\\
&=\left(\sum_{i=1}^{\infty}z_i^2+9-6z_k\right)^{1/2}\\
&\leq(45-3z_k)^{1/2}\leq6.7< 6.7+\frac{\varepsilon}{2}.
\end{split}
\end{equation*}
Hence, $z\in B_{\ell_2}(x_{k0},6.7+\frac{\varepsilon}{2})$.\\
And, $z$ has a nonnegative coordinate $z_{k'}$. Then,
\begin{equation*}
\begin{split}
\|z-x_{k^{'}0}\|&=\left(\sum_{j=1}^{\infty}z_j^2+64-16z_{k'}\right)^{1/2}\\
&\leq(100-16z_{k^{'}})^{1/2}\leq10< 10+\varepsilon.
\end{split}
\end{equation*}
Therefore, $z\in B_{\ell_2}(x_{k^{'}0},10+\varepsilon)$.\\ In the
latter case, since $\sum_{i=1}^{\infty}z_i^2$ and
$\sum_{j=1}^{\infty}z_j^2$ are convergent then $z_k\rightarrow0$
as $k\rightarrow\infty$ and $z_{k'}\rightarrow0$ as
$k'\rightarrow\infty$ therefore
\begin{equation*}
\|z-x_{k0}\|=\left(\sum_{i=1}^{\infty}z_i^2+9-6z_k\right)^{1/2}\leq(45-6z_k)^{1/2}<6.7+\frac{\varepsilon}{2},
\end{equation*}
and
\begin{equation*}
\|z-x_{k^{'}0}\|=\left(\sum_{j=1}^{\infty}z_j^2+64-16z_{k'}\right)^{1/2}\leq(100-16z_{k'})^{1/2}<10+\varepsilon,
\end{equation*}
 for any index $k$ and $k'$. On the other hand, any point $z\in S(O,6)$
with negative coordinate does not belong to the set
$\bigcup_{i=1}^{\infty}B_{\ell_2}(x_{i0},6.7)$ and
$\bigcup_{j=1}^{\infty}B_{\ell_2}(x_{j0},10)$, since for any
numbers $i$ and $j$
\begin{equation*}
\|z-x_{i0}\|=(45-6z_i)^{1/2}>6.7.
\end{equation*}
and
\begin{equation*}
\|z-x_{j0}\|=(100-16z_j)^{1/2}>10.
\end{equation*}
So, we have
\begin{equation*}
\begin{split}
&\inf\biggl\{l\ge0:~B_{\ell_2}(y_0,\sigma\sqrt{\theta})\subset\bigcup_{i=1}^{\infty}B_{\ell_2}(x_{i0},\rho_i\sqrt{\theta}+l)\biggr\}=\\
&\inf\biggl\{l\ge0:~B_{\ell_2}(O,6)\subset\bigcup_{i=1}^{\infty}B_{\ell_2}(x_{i0},6+l)\biggr\}=0.7,
\end{split}
\end{equation*}
and
\begin{equation*}
\begin{split}
&\inf\biggl\{l\ge0:~B_{\ell_2}(y_0,\sigma\sqrt{\theta})\subset\bigcup_{j=1}^{\infty}B_{\ell_2}(x_{j0},\rho_j\theta+l)\biggr\}=\\
&\inf\biggl\{l\ge0:~B_{\ell_2}(O,6)\subset\bigcup_{j=1}^{\infty}B_{\ell_2}(x_{j0},9+l)\biggr\}=1.
\end{split}
\end{equation*}
Therefore, the number
\begin{equation*}
\gamma=\min\biggl\{0.7,~~ 1\biggr\}=0.7,
\end{equation*}
is the value of the game.
\end{example}

 {\bf Conclusion.}
We considered a fixed duration pursuit-evasion problem with
countably many pursuers and one evader in the Hilbert space. The
controls of a group of pursuers are subject to geometric
constraints and the controls of the other pursuers and the evader
are subject to integral constraints. We fixed the index $i$ on
pursuers and constructed an admissible strategy for the pursuer
that guaranties it to capture the evader. Moreover, by taking
contribution from an auxiliary differential game under an
important assumption we guessed the value of the game and then we
proved the accuracy of our guess.


\bibliographystyle{amsmath}
\bibliographystyle{amsmath}

\end{document}